\documentclass{amsart}

\usepackage{amsfonts,amssymb}
\usepackage{float}
\usepackage{hyperref}
\usepackage{enumerate}
\usepackage[capitalise]{cleveref}
\usepackage{amsmath, amsthm, mathabx}
\usepackage{verbatim}
\usepackage[official]{eurosym}
\usepackage{geometry}
\usepackage{nomencl}
\usepackage{bm}

\usepackage{fancyhdr}
\usepackage{amscd}
\usepackage[english]{babel}
\usepackage{xcolor}

\usepackage{todonotes}
\usepackage{lipsum}
\usepackage{amsthm}

\newtheorem{theorem}{Theorem}[section]

\newtheorem{proposition}[theorem]{Proposition}

\newtheorem{lemma}[theorem]{Lemma}
\newtheorem{fact}[theorem]{Fact}

\theoremstyle{definition}
\newtheorem{definition}[theorem]{Definition}

\newtheorem{remark}[theorem]{Remark}

\newcommand{\enum}[2]{\begin{enumerate}[\hspace*{0.5cm}#1] #2 \end{enumerate}}

\newcommand{\C}[0]{\mathbb{C}}
\newcommand{\G}[0]{\mathcal{G}}
\newcommand{\Z}[0]{\mathbb{Z}}

\newcommand{\id}{\mathrm{id}}
\newcommand{\GL}{\mathrm{GL}}
\newcommand{\SL}{\mathrm{SL}}
\newcommand{\supp}{\mathrm{supp}}
\newcommand{\link}{\mathrm{link}}

\newcommand{\nin}[0]{\notin}
\newcommand{\precycle}[0]{precycle } 
\newcommand{\sq}[0]{\sqsubset}
\newcommand{\ta}[0]{\tilde{a}}
\newcommand{\R}[0]{R}

\newcommand{\ul}[0]{\underline}

\newcommand{\ut}[0]{\ul{t}}
\newcommand{\gp}[0]{\Gamma\mathcal{G}} 

\newtheorem*{main-thm}{Theorem \ref{main ring version}}
\newtheorem*{complex-corollary}{Corollary \ref{main complex version}}

\newtheorem*{thmA}{Proposition \ref{main lemma}}

\newtheorem{thmx}{Theorem}
\newtheorem{corx}[thmx]{Corollary}
\newtheorem{quex}[thmx]{Question}

\title{Linearity of graph products}

\author{Federico Berlai, Javier de la Nuez}

\address[Federico Berlai, Javier de la Nuez]{Department of Mathematics, UPV/EHU, Barrio Sarriena s/n, 48940, Leioa, Spain}

\email[Federico Berlai]{federico.berlai@gmail.com}
\email[Javier de la Nuez]{javier.delanuez@ehu.eus}

\begin{document}
\begin{abstract}
In this work we prove that, given a simplicial graph $\Gamma$ and a family $\mathcal{G}$ of linear groups over a domain $\R$, the graph product $\Gamma\mathcal{G}$ is linear over $\R[\underline t]$, where $\underline t$ is a tuple of finitely many linearly independent variables. As a consequence we obtain that any graph product of finitely many groups linear over the complex numbers is again a linear group over the complex numbers. This solves an open problem of Hsu and Wise~\cite[Problem 1, p.~258]{hsu.wise} in the case of faithful representations over $\C$.
\end{abstract}
\maketitle

\section{Introduction}

The study of linear groups has deep roots in group theory, and deciding whether a residually finite group is linear, or not, can often be a difficult task.
The first to systematically tackle this problem was Mal'cev~\cite{malcev}, who proved, among other things, that finitely generated linear groups are hopfian.

Over the course of the years, results concerning free products and amalgamated products \cite{M,Minty,pShalen,Wehrfritz}, and HNN extensions~\cite{MRV}, of linear groups have appeared. Shalen \cite{pShalen} in particular proved that amalgamated products of linear groups over $\C$, amalgamated over maximal cyclic subgroups, are linear over $\C( t)$, where $ t$ is an indeterminate, and therefore also linear over $\C$ (compare \cite[Lemma~1.2]{pShalen}, or Lemma \ref{function_field_embedding}), and that the class of linear groups over $\C$ is preserved under taking free products (compare \cite[Theorem~1]{pShalen}). This latter result was already proved by Wehrfritz \cite[Theorem~4]{Wehrfritz}, and was later generalized to faithful representations of free products over domains by Minty \cite{Minty}.

A major contribution to the theory was provided by Lubotzky \cite{Lub}, who proved that a finitely generated group is linear over $\C$ if and only if it has a $p$-congruence structure for almost all primes~$p$, that is, if and only if the group $G$ has a descending chain of finite-index normal subgroups $\{N_i\}_{i\geqslant 0}$ (with $N_0=G$) such that $\bigcap_{i}N_i=\{e_G\}$, with $N_0/N_i$ a finite $p$-group for all $i\geqslant 1$, and with a uniform number of generators for the quotients $N_i/N_j$, for all $j\geqslant i \geqslant 0$. 

\smallskip
In this work we are concerned with linearity of graph products of groups, which are a generalization, considered by Green \cite{green}, of free and direct products.
A graph product is defined in terms of a (simplicial) graph $\Gamma=(V,E)$, that is a graph with no multiple edges and no loops, and a collection of groups $\mathcal{G}=\{G_v\}_{v\in V}$, as the group
\[\Gamma\mathcal{G}:=\langle \ast_{v\in V}G_v\mid [G_v,G_u]=e\quad \forall\,\{v,u\}\in E\rangle.\]
The first two conditions in Lubotzky's characterization require the finitely generated linear groups to be residually $p$-finite, and it is known \cite[Theorem~5.6]{green} that the class of residually $p$-finite groups is preserved by taking graph products. On the other hand, it is not clear how the third condition is affected when taking graph products.

Taking inspiration from the work of Shalen~\cite{pShalen}, we prove the following theorem, giving a positive answer to~\cite[Problem 1, p.~258]{hsu.wise} for linear groups over the complex numbers (compare with Corollary~\ref{main complex version}):
  \begin{thmx}
  	\label{main ring version}\emph{Let $\Gamma=(V,E)$ be a finite graph and $\mathcal{G}=\{G_{v}\}_{v\in V}$ a collection of groups linear over an integral domain $R$. Then $\Gamma\G$ is linear over the ring of polynomials $R[\underline t]$, for some finite tuple of variables $\underline t$.}
  \end{thmx}
Notice that we do not require finite generation for the groups appearing in Theorem~\ref{main ring version}, nor any restriction on the cardinality. Moreover we are not restricting our attention to fields, as our approach works for any (not necessarily commutative) domain.
  
\smallskip  
As an immediate consequence of Theorem~\ref{main ring version}, we deduce:
  \begin{corx}
  	\label{main complex version}\emph{Let $\Gamma=(V,E)$ be a finite graph and $\mathcal{G}=\{G_{v}\}_{v\in V}$ a collection of groups linear over~$\C$. Then $\Gamma\G$ is linear over~$\C$.}
  \end{corx}

A group $G$ is called \emph{equationally noetherian} if the the set of solutions of any system $\mathcal{S}$ of equations over $G$ is equal to the set of solutions of a finite subsystem $\mathcal{S}_0\subseteq \mathcal{S}$. The terminology was introduced in \cite{BMR}, although the notion goes back to Bryant \cite{B}.
Groups that are linear over noetherian commutative unitary rings are equationally noetherian \cite[Theorem~B1]{BMR}. In particular, any group linear over a field is.
Therefore, from Theorem~\ref{main ring version} we deduce:
\begin{corx}\label{main noetherian version}
  	\emph{Let $\Gamma=(V,E)$ be a finite graph and $\mathcal{G}=\{G_{v}\}_{v\in V}$ a collection of groups linear over a commutative integral domain $\R$. Then $\Gamma\G$ is an equationally noetherian group.}
\end{corx}

This should be compared with \cite[Theorem~E]{Valiunas}, where it is proved that a graph product of equationally noetherian groups is equationally noetherian, whenever the girth of the graph (that is, the minimal length of any induced cycle) is at least five. In Corollary~\ref{main noetherian version} the restriction on the girth of the graph $\Gamma$ is not present, but the hypotheses on the starting groups are more restrictive: not all equationally noetherian groups are linear.

\medskip
Our methods, as well as the ones used to tackle amalgamated products of linear groups \cite{M, pShalen, Wehrfritz}, rely on trascendental extensions $R[\underline t]$ of the domain $R$, and therefore no conclusion can be drawn on the linearity over the domain $R$ itself. In particular, no result on linearity over the integers can be deduced by these methods.

On the other hand, HNN extensions of free abelian groups \cite{MRV} are linear over $\Z$. Moreover, it is known that graph products of some specific families of linear groups are again linear over~$\Z$: any graph product of subgroups of Coxeter groups \cite[Theorem~3.2]{hsu.wise}, and in particular any right-angled Artin group \cite{humphries} is linear over~$\Z$.

Thus, it is natural to ask (specializing \cite[Problem 5.1]{hsu.wise} to linearity over the integers):
\begin{quex}
Do finite graph products preserve linearity over $\Z$? Is the group $\SL_n(\Z)\ast \SL_n(\Z)$ linear over $\Z$?
\end{quex}

\subsubsection*{Acknowledgements}
Javier de la Nuez is, and Federico Berlai was, supported by the ERC grant PCG-336983, Basque Government Grant IT974-16, and Ministry of Economy, Industry and Competitiveness of the Spanish Government Grant MTM2017-86802-P. Federico Berlai is now supported by the Austrian Science Foundation FWF, grant no.~J4194. We are very grateful to Montserrat Casals Ruiz and Ilya Kazachkov for inspiring discussions over the topic of this work.

\section{Preliminaries}
  In this preliminary section, we collect known results that will be useful for our proof, and we establish notation.

Throughout this paper, a domain is a (not necessarily commutative) non-zero unitary ring without zero divisors, that is a non-zero unitary integral domain. The neutral element of a group $G$ is denoted by $e_G$, or by $e$ if the group is clear from the context.
Given a graph product $\Gamma\mathcal{G}$, the groups in $\mathcal{G}$ are called \emph{vertex groups}. An element $g\in \Gamma\mathcal{G}$ is called \emph{parabolic} if its conjugacy class intersects one of the vertex groups.

  We say that a graph $\Gamma=(V_\Gamma,E_\Gamma)$ is \emph{bipartite} if there exists a partition $V_\Gamma=V_0\sqcup V_1$ of the vertex set $V_\Gamma$ given by non-empty, disjoint subsets $V_0$ and $V_1$, such that any edge in $E_\Gamma$ connects a vertex in $V_0$ with one in $V_1$.

\smallskip
The \emph{complement graph} $\overline \Gamma=(V_{\overline\Gamma},E_{\overline\Gamma})$ is the graph whose vertex set is $V_{\overline\Gamma}=V_\Gamma$, and $\{v,u\}\in E_{\overline\Gamma}$ if and only if $\{v,u\}\notin E_\Gamma$.
  Given a subset $S\subseteq V_\Gamma$, the \emph{induced subgraph} on $S$ is the graph $\Delta$ such that $V_\Delta=S$, and there exists an edge $\{u_1,u_2\}\in E_\Delta$ if and only if there exists an edge $\{u_1,u_2\}\in E_\Gamma$.
  
  \begin{fact}\label{fact}
  	Any (finite) graph has a bipartite $2$-sheeted graph covering.
  \end{fact}

Given subsets $V_{0}$ and $V_{1}$ of $V$, we say that $V_{0}$ and $V_{1}$ are orthogonal if $V_{0}\cap V_{1}=\emptyset$ and for any $v_{0}\in V_{0}$ and $v_{1}\in V_{1}$ we have $\{v_{0},v_{1}\}\in E$.
  
  Given a subset of vertices $W\subseteq V$, we write $\Gamma\G_{\restriction W}$ for the subgroup of $\Gamma\mathcal{G}$ generated by $\bigcup_{w\in W}G_{w}$. It is isomorphic
  to $\Gamma_{\restriction W}\G_{\restriction W}$, where $\G_{\restriction W}=\{G_{w}\}_{w\in W}$, and $\Gamma_{\restriction W}$ is the subgraph of $\Gamma$ induced by the vertices $W$.
  	Notice that if $V_{0},V_{1},\dots, V_{k}\subset V$ are pairwise orthogonal, then
  	$\gp_{\restriction\bigcup_{j=0}^{k}V_{j}}\cong\bigoplus_{j=0}^{k}\gp_{\restriction V_{j}}$.
  
Given a non-trivial element $g\in\Gamma\mathcal{G}$, there is a unique minimal $W\subseteq V$ such that $g\in\gp_{\restriction W}$. We refer to $W$ as the \emph{support} of $g$, denoted by $\supp(g)$.
The element $g=g_{1}g_{2}\cdots g_{k}$, where $g_{i}\in G_{v_{i}}\setminus\{e_{G_i}\}$ for $1\leq i\leq k$, is called \emph{reduced} if there do not exist $1\leq i_{0}<j_{0}\leq k$ such that $v_{i_{0}}=v_{j_{0}}$ and $g_{j}$ commutes with $g_{i_0}$ for all $i_{0}<j<j_{0}$. This is a property of the word $g_1\dots g_k$ representing the element $g$, and not of the element itself. With a slight abuse of notation, whenever considering a reduced element, we will always implicitly consider an element and a reduced word representing it.  

We say that $g$ is \emph{cyclically reduced} if $\supp(g)\subseteq \supp(g^{h})$ for all $h\in\Gamma\mathcal{G}$. Every element $g\in\Gamma\mathcal{G}$ admits some cyclically reduced conjugate.
Any element $g\in\gp\setminus\{e\}$ is cyclically reduced if and only if it can be represented by a cyclically reduced word \cite[Definition~3.14]{green} (and hence all reduced words representing it are cyclically reduced).

\smallskip
We will need the following definition \cite[Section 2.10]{Goda}:
  \begin{definition}\label{definition_block}
  	An element $w$ of a graph product $\Gamma\mathcal{G}$ is a \emph{block} if the induced subgraph of $\bar \Gamma$ by $\supp(w)$ is connected. This happens if and only if the set of vertices $\supp(w)\subseteq\Gamma$ cannot be partitioned into pairwise orthogonal (in $\Gamma$), non-empty subsets.
  	
For any element $g\in\gp$, if $V_{1},V_{2},\dots, V_{k}$ are the sets of vertices of the connected components of $\supp(g)$ in $\bar{\Gamma}$, then $g$ can be expressed as a product of $k$ blocks $g=b_{1}b_{2}\dots b_{k}$, where $\supp(b_{j})=V_{j}$ for $1\leq j\leq k$. We refer to this as the \emph{block decomposition} of $g$. It can be shown that this expression is unique up to permutation. Notice that $b_{i}b_j=b_jb_i$ in $\Gamma\mathcal{G}$ for all $i\neq j=1,\dots,d$.
\end{definition}

\begin{lemma}\label{reduction}
  	Let $\Gamma$ and $\Delta$ be graphs, and let $\G=\{G_{v}\}_{v\in V_\Gamma}$ and $\mathcal{H}=\{H_{v}\}_{v\in V_\Delta}$ be families of groups.
  	Let $f\colon\bar{\Gamma}\to\bar{\Delta}$ be a graph covering with finite fibers. If for each pair $(w,v)\in V_\Delta\times V_\Gamma$ with $f(v)=w$ there exists an isomorphism
  	$\phi_{w,v}\colon H_{w}\to G_{v}$, then there is an embedding of $\Delta\mathcal{H}$ into $\Gamma\G$.
  \end{lemma}
  \begin{proof}
  	The maps $\phi_w\colon H_w\to\Gamma\mathcal{G}$ defined as  $\phi_w(h):=\prod_{f(v)=w}\phi_{w,v}(h)$ for all $h\in H_w$ are well-defined injective homomorphisms, and linearly extend to an injective homomorphism $\phi\colon \Delta\mathcal{H}\to\Gamma\G$, so that $\phi(h_1\dots h_n)=\phi_{w_1}(h_1)\dots\phi_{w_n}(h_n)$ where $h_1\dots h_n$ is a reduced expression in $\Delta\mathcal{H}$ and $h_i\in H_{w_i}$ for all $i$.
  	
  	First of all, notice that the maps $\phi_w$ are well defined. Indeed, as the map $f\colon \bar\Gamma\to\bar \Delta$ is a graph covering with finite fibers, we know that, given $w\in V_\Delta$, there are only finitely many $v\in V_{\Gamma}$ such that $f(v)=w$. Moreover, given any two vertices $v,\tilde v\in V_\Gamma$ such that $f(v)=f(\tilde v)=w$, we have that $\{v,\tilde v\}\in E_\Gamma$. Indeed, if this were not the case, as $f$ is a graph covering map we would obtain an edge $\{f(v),f(\tilde v)\}=\{w,w\}\in E_\Delta$, contradicting the fact that the graph $\Delta$ is simplicial. Thus $\langle G_v\mid f(v)=w\rangle\leqslant \Gamma\mathcal{G}$ is the direct product $\prod_{f(v)=w}G_v$, and therefore the image $\phi_w(h)$ is well defined. The homomorphisms $\phi_w$ are injective because the maps $\phi_{w,v}$ are isomorphisms.
  	
  	\smallskip
  	To prove that $\phi$ is a homomorphism we need to prove that whenever $\{w,\tilde w\}\in E_\Delta$ we have that $[\phi(h),\phi(\tilde h)]=e$ for all $h\in H_w$ and $\tilde h\in H_{\tilde w}$. We will prove that these are equivalent conditions. Let $v_1,\dots,v_n$ be the $f$-preimages of the vertex $w$, and $\tilde v_1,\dots,\tilde v_m$ be the $f$-preimages of the vertex $\tilde w$. As $\{w,\tilde w\}\notin E_{\overline \Delta}$, there cannot be any edge in $\overline \Gamma$ connecting a $v_i$ with $\tilde v_j$, because $f$ is a graph-covering map. That is $\{v_i,\tilde v_j\}\in E_\Gamma$ for all $i=1,\dots,n$ and for all $j=1,\dots,m$, and therefore $[\phi(h),\phi(\tilde h)]=e$ for all $h\in H_w$ and $\tilde h\in H_{\tilde w}$.
  	The opposite implication is proved with an analogous reasoning.
  	
  	\smallskip
  	Finally, the homomorphism $\phi\colon \Delta\mathcal{H}\to \Gamma\mathcal{G}$ is injective. Suppose we are given an element $h\in \Delta\mathcal{H}$ such that $\phi(h)=e$, and let $h=h_1\dots h_n$ be a reduced form for $h$, where $h_i\in H_{w_i}$ for all $i$. Therefore $e=\phi(h)=\phi_{w_1}(h_1)\dots\phi_{w_n}(h_n)$, and thus the right-hand side of the previous equation is not reduced. This is a contradiction with the assumption that $h_1\dots h_n$ is reduced, and with the fact that $[\phi(h_i),\phi( h_j)]=e$ in $\Gamma\mathcal{G}$ if and only if $[h_i,h_j]=e$ in $\Delta\mathcal{H}$. Therefore $h$ is the trivial element, and $\phi$ is injective.

  \end{proof}

  \begin{lemma}\label{non_zero_entries}
  	Let $G$ be a linear group over a ring $\R$. 
  	For any $n>1$ sufficiently large there exists an injective homomorphism $\lambda\colon G\to \GL_n(\R[X])$, where $X$ is a free variable over $\R$, such that $\lambda(g)_{i,j}\neq 0$ for all $g\in G\setminus\{e_G\}$ and for all $i\neq j=1,\dots,n$.
  \end{lemma}
  \begin{proof}
  Let $\rho\colon G\to \GL_n(\R)$ be a linear representation of the group $G$ over $\R$, with $n\geqslant 2$. Without loss of generality we can suppose that $\rho(g)$ is not a scalar matrix for all non-trivial $g\in G$. Indeed, if that were not the case, we could replace the representation $\rho$ with $\tilde \rho\colon G\to \GL_{n+1}(\R)$, defined for all $g\in G$ by $\tilde \rho(g)=\bigl(\begin{smallmatrix}1&0\\ 0&\rho(g) \end{smallmatrix}\bigr)$. The same argument allows to assume that $n$ is arbitrary large.
  	
  	For any pair $i\neq j$ of indices in $\{1,\dots,n\}$, let $E_{i,j}$ be the matrix whose only non-zero entry is the entry $(i,j)$, which is defined to be $1$. Let $F_{i,j}(X):=I_n+XE_{i,j}\in GL_{n}(\R[X])$, where $X$ is a free variable over $R$ that commutes with the elements of the ring. Notice that $F_{i,j}(X)^{-1}=F_{i,j}(-X)$ and $F_{i,j}(0)=I$.

  	Given a matrix $A=(a_{r,s})\in \GL_n(\R)$, the conjugation of $A$ by $F_{i,j}(X)$ results in the matrix
  	\begin{equation*}
  		(b_{r,s})_{r,s}:=F_{i,j}(X)^{-1}AF_{i,j}(X)=A^{F_{i,j}(X)}\in \GL_n(\R[X]),
  	\end{equation*}
  	given by
  	\begin{equation}\label{new_coefficients}
  		b_{r,s}=\begin{cases}
  			a_{r,s}&\quad\text{if }r\neq i\text{ and }s\neq j;\\
  			a_{i,s}-a_{j,s}X&\quad\text{if }r= i\text{ and }s\neq j;\\
  			a_{r,j}+a_{r,i}X&\quad\text{if }r\neq i\text{ and }s= j;\\
  			a_{i,j}+(a_{i,i}-a_{j,j})X-{a_{j,i}X^{2}}&\quad\text{if }r= i\text{ and }s= j.
  		\end{cases}
  	\end{equation}
  	Notice that the coefficient $b_{r,s}$ can only be equal to zero if the corresponding $a_{r,s}$ was also equal to zero.
  	In other words, if we denote by $ND_{0}(A)$ the collection of all pairs of indices $(i,j)\in\{1,\dots, n\}^{2}\setminus\{(k,k)\}_{1\leq k\leq n}$ such that $a_{i,j}=0$, then we have $ND_{0}(B)\subseteq ND_{0}(A)$.
  	
  	Pick a collection of independent polynomial variables $\{t_{i,j}\,|\,1\leq i,j\leq n\}$, and define
  	\begin{equation}\label{product}
  		F(\ut):=\prod_{i=1}^n\Bigl(\prod_{\substack{j=1\\j\neq i}}^nF_{i,j}(t_{i,j}) \Bigr)\in\GL_n(R[\ut]).
  	\end{equation}
  	We claim that, given any non-scalar matrix $A\in \GL_n(\R)$,
  	if $ND_{0}(A)\neq\emptyset$
  	then $ND_{0}(A^{F(\underline{t})})\subsetneq ND_{0}(A)$.
  	To begin with, notice that in order to show this it suffices to prove that for any such $A$ there exist distinct $i,j\in\{1,\dots, n\}$ such that $ND_{0}(A^{F_{i,j}(X)})\subsetneq ND_{0}(A)$.
  	Indeed, if we write $\ut$ as the sequence $\ut',t_{i,j},\underline t''$, where $\underline t'$ includes $t_{k,l}$ for all indices $(k,l)$ appearing to the left of $(i,j)$ in the product appearing in Equation \eqref{product}, and $\underline t''$ all those appearing to the right, then $F(0,t_{i,j},0)=F_{i,j}(t_{i,j})$, since $F_{r,s}(0)=I$, and therefore:
  	\[
  	ND_{0}(A^{F(\ut)})\subseteq ND_{0}(A^{F(0,t_{i,j},0)}))=ND_{0}(A^{F_{i,j}(t_{i,j})})\subsetneq ND_{0}(A).
  	\]
  	So let $A\in\GL_n(\R)$ be a non-scalar matrix. Consider first the case in which $A$ is a diagonal matrix. Since $A$ is not scalar, there exist entries $(i,i)$ and $(j,j)$ such that $a_{i,i}\neq a_{j,j}$. Let
  	$B=(b_{r,s}):=A^{F_{i,j}(X)}$. It follows from the fourth line of Equation \eqref{new_coefficients} that $b_{i,j}$ contains a non-zero linear term. In particular $b_{i,j}\neq 0$ and thus $ND_{0}(B)\subsetneq ND_{0}(A)$.
  	
  	If $A$ is not a diagonal matrix and has non-diagonal zero entries, then there exits at least a row, or a column, with a non-diagonal zero entry, and a non-diagonal non-zero entry. Assume without loss of generality that the $j$-th column satisfies this property, and take $a_{i,j}= 0$ and $a_{k,j}\neq 0$ in this column, where $i\neq j$ and $k\neq j$. If $(b_{r,s})_{r,s}=A^{F_{i,k}(X)}$ then
  	by Equation \eqref{new_coefficients} again
  	we have $b_{i,j}=a_{i,j}-a_{k,j}X=a_{k,j}X\neq 0$.
  	
  	Therefore, as claimed, $ND_{0}(A^{F})\subsetneq ND_{0}(A)$. By applying the argument $n^2-n$ times and using a new set of independent variables at each step, we deduce that there exists a matrix $H\in \GL_n(\R[\underline s])$ such that $ND_{0}(\rho(g)^{H})=\emptyset$ for all $g\in G\setminus\{e_G\}$, where $\underline s$ is a tuple of polynomial variables independent over $\R$.
  	
  	\smallskip
  	By replacing the finitely many variables in $\underline{s}$ with appropriate powers of a single variable, we obtain the statement.
  \end{proof}

  The following definition will play a pivotal role in the proof of Theorem \ref{main ring version}.
  \begin{definition}\label{defJV}
  	Let $\Gamma=(V_\Gamma,E_\Gamma)$ be a finite simplicial graph, consider	 a collection of groups $\{G_{v}\}_{v\in V_\Gamma}$ with injective homomorphisms $\lambda_{v}\colon G_v\to \GL_N(\R)$, for $N$ big enough.
  	We say that a family $\{J_{v}\}_{v\in V}$ is a \emph{well-placed collection} of supports for $\{\lambda_{v}\}_{v\in V}$ if the following properties are satisfied:
  	\enum{(a)}{
  		\item\label{conda} $J_{v}\cap J_{w}=\emptyset$ for $\{u,v\}\in E_\Gamma$;
  		\item\label{condb} $\lvert J_{v}\cap J_{w}\rvert>2$ for $\{u,v\}\notin E_\Gamma$;
  		\item\label{condc} $J_{u}\cap J_{v}\cap J_{w}=\emptyset$ for every three distinct $u,v,w\in V_\Gamma$.
  	}
  	Moreover, for all $g\in G_v\setminus\{e\}$ we require that:
  	\enum{(i)}{
  		\item $\lambda_v(g)_{i,j}\neq 0$ if $i,j\in J_{v},i\neq j$;
  		\item $\lambda_v(g)_{i,j}=\delta_{i,j}$ otherwise.
  	}
  \end{definition}
  
  Using Lemma \ref{non_zero_entries}, we can construct well-placed collections of supports for a finite graph product of linear groups. Indeed, given a finite graph $\Gamma=(V,E)$, for any $v\in V$ choose $n_v\geqslant \lvert\link_\Gamma(v)\rvert$ such that there exists an injective homomorphism $\lambda_v\colon G_v\to \GL_{n_v}(\R)$ such that $\lambda_v(g_v)_{i,j}\neq 0$ for all $g_v\in G_v$ and for all $i\neq j$ from $1$ to $n_v$. For $N$ big enough, depending on $n_{v_1},\dots,n_{v_{\lvert\Gamma\rvert}}$, we can choose subsets $J_{v_i}\subseteq \{1,\dots,N\}$, and with this choice replace the injective homomorphism $\lambda_v\colon G_v\to \GL_{n_v}(\R)$ with $\lambda_v\colon G_v\to \GL_{N}(\R)$, so that the constraints of Definition \ref{defJV} are satisfied.

  \medskip
  Let $T$ be an indeterminate, and consider the matrix
  \begin{align*}
  	D=\bigl(d_{i,j}(T)\bigr)_{1\leq i,j\leq 2}=
  	\begin{pmatrix}
  		T^{2} & T+1 \\
  		T-1 & 1
  	\end{pmatrix}
  	\in \SL_{2}(R[T]),
  \end{align*}
  with inverse
  \begin{align*}
  	D'=(d'_{i,j}(T))_{1\leq i,j\leq 2}=
  	\begin{pmatrix}
  		1 & 1-T \\
  		-1-T & T^{2}
  	\end{pmatrix}.
  \end{align*}
  For each unoriented edge $\{u,v\}$ in $\bar{E}$ we distinguish a pair of indices $\{r^{1}_{u,v},r^{2}_{u,v}\}\subset J_{u}\cap J_{v}$.
  Notice that $r^{1}_{u,v}=r^{1}_{v,u}$ and $r^{1}_{u,v}=r^{2}_{v,u}$.
  
  We now expand $\R$ with a transcendental tuple of elements $\ut= \{t_{w}^{v}\mid v\in V,\, w\in \link_{\overline \Gamma}(v)\}$, that is, a tuple of polynomial variables. We impose that the elements in the tuple pairwise commute, and commute with the elements of the domain $R$. Note that $t^{v}_{w}\neq t^{w}_{v}$.
  For each $v\in V$ consider the sub-tuple $\ut^{v}=(t^{v}_{w})_{w\in \link_{\overline\Gamma}(v)}$, and let
  $T_{v}\in\GL_N(\R(\ut^{v}))$ be the 
  matrix defined by
  \begin{equation}\label{definitionTs_blue}
  	(T_v)_{k,k'}:=\begin{cases}
  		1 &\quad\text{if }k=k'\nin \bigcup_{w\in \link_{\overline\Gamma}(v)}\{r^{1}_{v,w},r^{2}_{v,w}\}\\
  		d_{\sigma_{v}(\delta),\sigma_{v}(\epsilon)}(t^{v}_{w}) &\quad\text{if $(k,k')=(r^{\delta}_{v,w},r^{\epsilon}_{v,w})$ for $w\in \link_{\bar{\Gamma}}(v)$ and $\delta,\epsilon\in\{1,2\}$}\\
  		0&\quad\text{otherwise;}
  	\end{cases}
  \end{equation}
  where $\sigma_{v}=\id_{\{1,2\}}$ if $v\in V_{0}$ and $\sigma_{v}=(1\,\,2)$ if $v\in V_{1}$, and $d_{\sigma_{v}(\delta),\sigma_{v}(\epsilon)}(t^{v}_{w})$ is a polynomial entry of the matrix $D$ evaluated at $t^v_w$.

  \begin{lemma}\label{commutation}
  	The maps $\widetilde{\lambda}_v\colon G_v\to \GL_N(\R[\ut])$ sending $g\in G_v$ to $T_{v}\lambda_{v}(g){T_{v}^{-1}}$ extend to a homomorphism $\lambda\colon\gp\to \GL_{N}(\R[\ut])$. Moreover, $\lambda(g)=e$ if and only if $\lambda(b)=e$ for every block $b$ of~$g$.
  \end{lemma}
  \begin{proof}
  	The group $\GL_N(\R)$ is the group of $\R$-module automorphisms of the free $\R$-module $\R^N$. Let $\{m_i\mid i=1,\dots,N\}$ be a free basis of the $\R$-module $\R^N$. Any set $J_v$ induces a decomposition of $\R^N=M_{\overline v}\oplus M_v$, where $M_v=\{m_j\mid j\in J_v\}$ and $M_{\overline v}=\{m_j\mid j\notin J_v\}$. The element $T_v\in \GL_N(\R[\underline t^v])$ is such that its restriction to $M_{\overline v}$ is the identity of $M_{\overline v}$. Moreover, $\lambda_v(g)$ is also the identity of $M_{\overline v}$ for all $g\in G_v$. Thus $\widetilde \lambda_v(g)=T_v^{-1}\lambda_v(g)T_v$ will also be the identity as an $\R$-module automorphism of~$M_{\overline v}$.
  	
  	Both claims follow from this and from the fact that $J_{v}\cap J_{w}=\emptyset$ whenever $\{v,w\}\in E$ is an edge of the graph~$\Gamma$.
  \end{proof}

  The following proposition implies Theorem \ref{main ring version} of the introduction. Its proof will be provided in Section \ref{section_trace}.
  \begin{proposition}\label{main lemma}
  	Given a graph $\Gamma$ with $\overline \Gamma$ bipartite and any non-parabolic block $g\in\mathcal{G}$, we have that the trace $tr\bigl(\lambda(g)\bigr)$ is not an element of $R$.
  \end{proposition}
  Using Proposition \ref{main lemma}, we now give a proof of Theorem \ref{main ring version} and of Corollary \ref{main complex version}.
  \begin{main-thm}
  	\emph{Let $\Gamma=(V,E)$ be a finite graph and $\mathcal{G}=\{G_{v}\}_{v\in V}$ a collection of groups linear over a domain $R$. Then $\Gamma\G$ is linear over the ring of polynomials $R[\underline t]$, for some finite tuple of variables~$\underline t$.}
  \end{main-thm}
  \begin{proof}
  	Using Lemma \ref{non_zero_entries}, as described after Definition \ref{defJV}, it is easy to construct
  	a collection $\{\lambda_{v}\colon G\to \GL_{N}(\R)\}_{v\in V}$ of faithful linear representations admitting a system of well-placed supports.
  	Moreover, by Fact \ref{fact} and Lemma \ref{reduction}, it is sufficient to prove the result for graphs such that $\overline \Gamma$ is bipartite.
  	Let $\lambda\colon G\to \GL_{N}(R[\ut])$ be the homomorphism provided by Lemma \ref{commutation}. To establish injectivity of $\lambda$ it suffices to show that $\lambda(g)\neq e$ for all $g\in \Gamma\mathcal{G}\setminus\{e\}$.
In view of Lemma~\ref{commutation}, we may assume without loss of generality that $g$ consists of a single block.
  	
  	If $g$ is a parabolic element then, up to conjugation, $g\in G_{v}$ for some $v\in V$, and therefore $\lambda(g)\neq e$ because $\lambda_{v}$ is injective. On the other hand, if $g$ is not parabolic then the equality $\lambda(g)=e$ would imply that $tr(\lambda(g))\in \R$. Since the trace is invariant under conjugation, we may assume that $g$ is cyclically reduced, thereby contradicting Proposition~\ref{main lemma}.
  \end{proof}
  
  For a proof of the following fact, see for instance \cite[Lemma~1.2]{pShalen}.
  \begin{lemma}\label{function_field_embedding}
  	Let $\ut$ be a finite tuple of transcendental elements over $\C$. Then, the rational function
  	field $\C(\ut)$ is isomorphic to a subfield of $\C$.
  \end{lemma}
  Combining Lemma \ref{function_field_embedding} and Theorem \ref{main ring version}, we obtain:
  
\begin{complex-corollary}
\emph{Let $\Gamma=(V,E)$ be a finite graph and $\mathcal{G}=\{G_{v}\}_{v\in V}$  a collection of groups linear over $\C$. Then $\Gamma\G$ is linear over $\C$.}
  \end{complex-corollary}

\section{Cycles and precycles} \label{section cycles}
In this section we fix a cyclically reduced block, and we develop some tools that will be used in Section~\ref{section_trace} for the proof of Proposition~\ref{main lemma}.

\medskip
By Fact \ref{fact} and Lemma \ref{reduction}, we can assume that the graph $\overline \Gamma$ is bipartite, that is there is a non-trivial partition $(V_{0},V_{1})$ of $V$ such that any edge in $\bar{E}$ connects a vertex from $V_{0}$ to one in $V_{1}$.
  
  Fix now some cyclically reduced block $g\in\Gamma\mathcal{G}$ with reduced normal form
  \begin{equation}\label{eq_normal_form_forg}
  	g=g_{1}g_{2}\cdots g_{m},\qquad g_{l}\in G_{v_{l}}\quad\forall\, l=1,\dots,m.
  \end{equation}
  Notice that two elements $g_i$ and $g_j$ appearing in this reduced expression for $g$ may belong to the same vertex group, that is $v_i=v_j$. Nevertheless, two such occurrences cannot be joined together by successive switches of syllables, because the expression in Equation \eqref{eq_normal_form_forg} is reduced.

  Given an interval $\{1,\dots, m\}$ and two integers $l,l'\in\{1,\dots, m\}$, we define the \emph{cyclic interval} $(l,l')$ to be
  \begin{equation}\label{cyclic_interval}
  	(l,l')=\begin{cases}
  		\{l+1,\dots,l'-1\}&\quad\text{if }\quad l<l';\\
  		\{1,\dots,m\}\setminus\{l\}&\quad\text{if }\quad l=l';\\
  		(1,l')\cup(l,m)&\quad \text{if }\quad l>l'.
  	\end{cases}
  \end{equation}
  Given a sequence $(i_1,\dots,i_k)$ of integer numbers, we say that two entries $i$ and $i'$ are \emph{cyclically consecutive} (in such sequence) if there exists $j\in\{1,\dots,k-1\}$ such that $i=i_j$, $i'=i_{j+1}$, or if $i=i_k$ and $i'=i_1$.
  
  \begin{definition}
  	We call a sequence $\bar{i}$ of indices $i_{1},i_{2},\dots, i_{k}$ in $\{1,\dots, m\}$ a \emph{path} if
  	the sequence of vertices $v_{i_{1}},v_{i_{2}},\dots v_{i_{k}}$ is a path in $\overline{\Gamma}$ and either:
  	\enum{(i)}{
  		\item the sequence $(i_{j})_{j=1}^{k}$ is strictly increasing; or
  		\item there is some $1< j_{0}\leq k$ such that
  		the sequence $(i'_{j})_{j=1}^{k}$ is strictly increasing, where
  		\begin{equation*}
  			i'_{j}=\begin{cases}i_{j}&\qquad\text{if }j<j_{0};\\
  				i_{j}+m&\qquad\text{otherwise.}
  			\end{cases}
  		\end{equation*}
  	}
  	If the sequence of associated vertices, up to cyclic permutation, forms a cycle in $\bar{\Gamma}$, then we say that $\bar{i}$ is a \emph{cycle}.
  	Moreover, we say that $\bar{i}=(i_{1},i_{2},\dots, i_{k})$ is a \emph{precycle} if it can be obtained from a cycle by removing some entries. In these two situations we will implicitly identify two sequences differing only by a cyclic permutation.
  	
  \end{definition}
  Given two precycles $(i_{1},i_{2},\dots,i_{k})$ and $(i'_{1},i'_{2},\dots, i'_{k'})$,
  we say that the latter is a \emph{refinement} of the former if $\{i_{j}\}_{j=1}^{k}$ can be obtained from $\{i'_{j}\}_{j=1}^{k'}$ by removing some entries. Note that, by definition, a precycle which is maximal for refinement is a cycle.

  For $l$ and $l'$ in $\{1,\dots, m\}$, we write $l\sqsubset l'$ if there is a sequence of indices $l=l_{0},l_{1},\dots, l_{k}=l'$ that forms a path.
  The following observation is a consequence of the fact that all cyclic permutations of the word $g_{1}g_{2}\dots g_{m}$ are cyclically reduced.
  \begin{lemma}\label{no immediate repetitions}
  	Let $ l, l'\in\{1,\dots, m\}$. If $v_{l}=v_{l'}$, then $l\sqsubset l'$ and $l'\sqsubset l$.
  \end{lemma}
  \begin{proof}
  	If $l=l'$ there is nothing to prove.
  	
  	Assume that $l<l'$. As the expression $g=g_1\dots g_m$ is reduced, there exists $l''\in(l,l')$ such that $\{v_{l},v_{l''}\},\{v_{l'},v_{l''}\}\nin E_{\Gamma}$, and therefore $l\sqsubset l'$. Otherwise, it would be possible to permute $g_{l'}$ next to $g_{l}$ and then merge the two syllables.
  	
  	The fact that $g=g_{1}g_{2}\dots g_{m}$ is cyclically reduced implies that there exists $\tilde l\in [1,\dots,l)\sqcup(l',\dots,m]$ such that
  	$\{v_{l},v_{\tilde l}\},\{v_{l'},v_{\tilde l}\}\nin E_{\Gamma}$. Thus we conclude that $l'\sqsubset l$ too.
  \end{proof}

  \begin{remark}\label{singletons are precycles}
  	In particular if $l=l'$ the lemma above can be reformulated as saying that any singleton is a precycle.
  \end{remark}

  \medskip
  Fix an enumeration $\{u_{1},u_{2},\dots ,u_{|V|}\}$ of the vertex set $V_\Gamma$, in such a way that $u_1$ appears in the support of $g$ in Equation \eqref{eq_normal_form_forg}.
  
  For any $1\leq n\leq \lvert V\rvert$, let $L_{n}:=\{l\in\{1,\dots, m\}\mid \,v_{l}=u_{n}\}$, and
  \[M_{n}:=\begin{cases}
  	\bigl\{l\in L_{n+1}\mid \exists\, j\neq j'\in K_n\text{ cyclically consecutive},\ j\sqsubset l\sqsubset j' \bigr\}&\text{if }\lvert K_n\rvert\geqslant 2\\
  	\bigl\{l\in L_{n+1}\mid j\sqsubset l\sqsubset j \bigr\}&\text{if } K_n=\{j\}.
  \end{cases}\]
  Finally, let $K_{n}\subset\{1,\dots, m\}$ be defined as follows:
  \enum{(i)}{
  	\item  $K_{1}=L_{1}$;
  	\item  $K_{n+1}=K_{n}\cup M_{n+1}$.
  }
  For notational purposes, let $M_0=K_0=L_0=\emptyset$.
  
  \medskip
  By the choice of the enumeration of the vertices in $\Gamma$, the vertex $u_1$ appears as one of the vertices for the elements in the reduced expression of $g$ of Equation \eqref{eq_normal_form_forg}. Therefore $L_1\neq \emptyset$.

  The following lemma is trivially true for $n=0$ or $n=1$. For $n\geqslant 2$, it is a consequence of the definition of $K_{n}$ and of the fact that the intersection $K_{n-1}\cap L_n$ is empty.
  \begin{lemma}\label{lemma_equality}
  	$K_{n}\cap L_{n}=M_{n}$ for all $n$.
  \end{lemma}

  \begin{lemma}\label{precycle}
  	$K_{n}$ is a precycle for all $n$. Moreover, $K_{\lvert V\rvert}$ is a cycle.
  \end{lemma}
  \begin{proof}
  	We first prove that $K_{n}$ is a precycle for all $n$. If $K_n$ is a singleton there is nothing to prove. So, suppose that $\lvert K_n\rvert\geqslant 2$.
  	First, notice that $K_1=L_1$ is a precycle, because it is a refinement of the precycle $\bar i=(1,2,\dots,m)$, and the latter can be extended to a cycle as the expression for $g$ appearing in Equation \eqref{eq_normal_form_forg} is reduced.

  	Given two cyclically consecutive indices $l,l'$ in $K_{n}$, consider the collection $\mathcal{I}_{l,l'}=\{j\in L_{n+1}\,|\,l\sq j \sq l'\}$. Lemma \ref{no immediate repetitions} implies
  	$\mathcal{I}_{l,l'}\cup\{l,l'\}$ forms a path $l\sq j_{1}\sq j_{2}\sq\dots\sq j_{s}\sq l'$. Since
  	$K_{n+1}$ is the union of $K_{n}$ and the $\mathcal{I}_{l,l'}$ associated with each of the pairs $l,l'\in K_{n}$ as above, from the assumption that
  	$K_{n}$ is a \precycle it follows that $K_{n+1}$ is one as well. The conclusion follows by induction.

  	\medskip
  	We now prove that $K_{\lvert V\rvert}=(i_1,\dots,i_k)$ is a cycle. Suppose it is not, so that there exists a cycle of which $K_{\lvert V\rvert}$ is a refinement, and let $i_\star$ be an entry of the cycle that is not present in $K_{\lvert V\rvert}$. Suppose that $i_n\sqsubset i_\star \sqsubset i_{n+1} $ for some $n\in \{1,\dots, k\}$, and that $v_{i_\star}\in L_j$ for some $j\in \{1,\dots,m\}$. Therefore, by definition of the set $K_{n+1}$ (and of the set $K_{\lvert V \rvert}$), it would follow that the index $i_\star\in K_{n+1}\subseteq K_{\lvert V\rvert}$, contradicting the assumption that $i_\star\notin K_{\lvert V\rvert}$. Thus we reached a contradiction, and therefore $K_{\lvert V\rvert}$ is a cycle.
  \end{proof}
  
  The precycles $K_{n}$ are \textquotedblleft optimal\textquotedblright\ in the following sense:
  
  \begin{lemma}\label{optimality}
  	Consider $0\leq n\leq \lvert V\rvert-1$, a precycle $J$, and assume that $J\cap (\bigcup_{i=0}^{n}L_{i})=K_{n}$. Then $J\cap L_{n+1}\subseteq M_{n+1}\subseteq K_{n+1}$.
  \end{lemma}
  \begin{proof}
  	For $n=0$ the result is clear, since $L_{1}=K_{1}$, and therefore $J\cap L_1\subseteq L_1=M_1=K_1$.
  	
  	So, suppose that $n\geq 1$ and let $j\in J\cap L_{n+1}$. Take cyclically consecutive indices $l,l'\in K_{n}$ such that $j\in(l, l')$, where $(l,l')$ is the cyclic interval defined in Equation \eqref{cyclic_interval}.
  	Since in addition to this $K_{n}\cup\{j\}\subset J$ and $J$ is a precycle, we have $l\sq j\sq l'$. The definition of $K_{n+1}$ then implies that $j\in K_{n+1}$.
  \end{proof}

\section{Traces and degrees}\label{section_trace}
In this final section, we apply the tools developed in Section~\ref{section cycles} to prove that the trace of the matrix $\lambda(g)$ is not in $R$, thus proving Proposition~\ref{main lemma}.

Let
  \begin{equation*}\label{eq_definition_matrix_entries}
  	\lambda(g)=(c_{i,j})
  	\in\GL_N(\R[\underline{t}]),\qquad\lambda_l(g_{l})=(a^{l}_{i,j})
  	\in \GL_{N}(\R),\qquad \lambda_{l}(g_{l})^{T_{v_{l}}}=(\ta_{i,j}^{l})\in \GL_{N}(\R[\underline{t}^{v_l}]).
  \end{equation*}
  We have that
  \begin{equation}\label{eq_diagonal_element}
  	c_{i,i}=\sum_{1\leqslant k_1,\dots,k_{m-1}\leqslant N}
  	\tilde a_{i,k_1}^{1}\tilde a_{k_1,k_2}^{2}\dots \tilde a_{k_{m-1},i}^{m}.
  \end{equation}
  \begin{remark} \label{obs_leading}
  	For any $1\leq i,j\leq n$, the entry $\ta_{i,j}^{l}$ is a polynomial in
  	$R[\ul t^{v}]$ of degree at most four. It has degree exactly four precisely when there exist $w,w'\in \link_{\bar{\Gamma}}(v_{l})$ such that $(i,j)=(r^{\delta_{l}(1)}_{v_{l},w},r^{\delta_{l}(2)}_{v_{l},w'})$ and $a^{l}_{i,j}\neq 0$, where $(\delta_{l}(1),\delta_{l}(2))$ is equal to $(1,2)$ in case $v_{l}\in V_{0}$, and to $(2,1)$ otherwise.
  	In this case the homogeneous component of degree four of $\ta_{i,j}^{l}$ is equal to $a^{l}_{i,j}(t^{v_{l}}_{w})^{2}(t^{v_{l}}_{w'})^{2}$.
  \end{remark}

  Let $\mathcal{X}$ be the collection of all sequences
  $\bar k=(k_{1},\dots, k_{m-1},k_m)\in\{1,\dots, N\}^{m}$
  such that
  $a^{v_l}_{k_{l},k_{l+1}}\neq 0$
  for all $1\leq l\leq m$, where $k_{m+1}$ is intended as $k_1$.
  Given a sequence $\bar{k}\in\{1,2,\dots, N\}^{m}$, let $P(\bar{k}):=\prod_{l=1}^{m}\ta^{l}_{k_{l},k_{l+1}}$. As $R$ is an integral domain, $P(\bar{k})\neq 0$ if and only if $\bar{k}\in\mathcal{X}$, and following Equation \eqref{eq_diagonal_element} we have \begin{equation}\label{equation_trace}
  	tr\bigl(\lambda(g)\bigr)=\sum_{i=1}^Nc_{i,i}=\sum_{\bar{k}\in\mathcal{X}}P(\bar{k}).
  \end{equation}
  Given $v\in V$ and $\bar{k}\in \mathcal{X}$, let $D^{v}(\bar{k})\in\mathbb{N}$ stand for the total degree of the polynomial $\prod_{v_{l}=v}\ta^{l}_{k_{l},k_{l+1}}\in\R[\ut^{v}]$.
  \begin{remark} \label{leading monomial}
  	For any monomial $\mu$ appearing in $P(\bar{k})$ and all $v\in V$ the sum of the exponents of variables of the form $t^{v}_{w}$ in $\mu$, that is the $\ut^{v}$-degree of $\mu$, is at most $D^{v}(\bar{k})$. Moreover, the equality is achieved for at least one of such $\mu$. We refer to any such $\mu$ as a leading monomial of $P(\bar{k})$.
  \end{remark}

  Let $D(\bar k)=\bigl(D_{u_1}(\bar k),\dots, D_{u_m}(\bar k)\bigr)\in \mathbb{N}^{\lvert V\rvert} $.
  We let $\prec$ stand for the lexicographical order on $\mathbb{N}^{|V|}$: we say that $\ul e\prec \ul e'$ if and only if there is $i_{0}\in\{1,2,\dots,\lvert V\rvert\}$ such that
  $e_{i}=e'_{i}$ for all $i<i_{0}$ and
  $e_{i_{0}}<e'_{i_{0}}$.

  \begin{lemma}
  	\label{monomial domination}Let $\bar{k},\bar{k}'\in\mathcal{X}$ and assume that
  	$D(\bar{k}')\prec D(\bar{k})$. Then no leading monomial of $P(\bar{k})$ is a scalar multiple of a monomial from $P(\bar{k}')$.
  \end{lemma}
  \begin{proof}
  	Let $(s_{v})_{v\in V}$ be a new tuple of variables. The homomorphism $\psi\colon \R[\underline{t}]\to \R[\underline{s}]$ fixing $\R$ and sending $t^{v}_{w}$ to $s_{v}$, for all $w\in \link_{\overline \Gamma}(v)$,
  	sends any leading monomial of $P(\bar{k})$ to a scalar multiple of
  	$\prod_{v\in V}s_{v}^{D^{v}(\bar{k})}$, while it sends any monomial $\nu$ of $P(\bar k')$ to a scalar multiple of $\prod_{v\in V}s_{v}^{e_{v}}$, where $e_v\leq D_{v}(\bar k')$ for all $v\in V$.
  	 This implies $(e_{u_{i}})_{1\leq i\leq |V|}\precneq D(\bar{k})$. This implies that $\psi(\nu)$ is not a scalar multiple of the image by $\psi$ of any leading monomial of $P(\bar{k})$. 
  \end{proof}

  Let $I(\bar{k})\subset\{1,\dots, m\}$ be the collection of all indices $l\in\{1,2,\dots m\}$ such that $\{k_{l},k_{l+1}\}\subseteq J_{v_{l}}$, where indices are taken mod $m$, and order such set from the smallest to the biggest natural number.
  \begin{lemma}\label{jump set}
  	Let $\bar{k}\in\mathcal{X}$. Then $k_{l}=k_{l+1}$ for any $l\nin I(\bar{k})$. Moreover, the set $I(\bar{k})$ is a precycle.
  \end{lemma}
  \begin{proof}
  	The first statement is immediate from the fact that for any $1\leq j\leq m$ and $k,k'\in \{1,\dots,N\}$ either $\tilde{a}^{j}_{k,k'}=0$ or $k=k'\nin J_{v_{j}}$ or $k,k'\in J_{v_{j}}$.  Given cyclically consecutive indices $l$ and $l'$ in $I(\bar{k})$ we thus have $k_{l'}=k_{l+1}\in J_{v_{l}}\cap J_{v_{l'}}\neq\emptyset$, which implies that
  	either $v_{l}=v_{l'}$ or $\{v_{l},v_{l'}\}\in\bar{E}$ and hence that $l\sq l'$ either by Lemma \ref{no immediate repetitions}
  	or by definition $l\sq l'$.
  \end{proof}

  The following lemma provides algebraic meaning to the combinatorial notions of Section \ref{section cycles}:
  \begin{lemma}\label{upper bound}
  	The following statements hold:
  	\enum{(a)}{
  		\item \label{item1} For any $\bar{k}\in\mathcal{X}$ and for any $1\leq i\leq |V|$, we have that  $\lvert D_{u_{i}}(\bar{k})\rvert \leq 4\bigl\lvert L_{i}\cap I(\bar{k})\bigr\rvert$.
  		\item \label{item2} Let $I\subset\{1,\dots, m\}$ be a cycle. Then there exist a unique $\bar{k}\in\mathcal{X}$ such that $I(\bar{k})=I$ and $D_{u_{i}}(\bar{k})= 4\lvert L_{i}\cap I\rvert$ for all $1\leq i\leq |V|$.
  	}
  \end{lemma}
  \begin{proof}
  	For any $1\leq l\leq m$, if $u_{i}\neq v_{l}$ then $D^{u_{i}}(\ta_{k_{l},k_{l+1}})=0$, while in case $u_{i}=v_{l}$ the inequality $D_{v_{l}}(\ta_{k_{l},k_{l+1}})\leq 4$ holds. Condition $\eqref{item1}$ immediately follows.
  	
  	It remains to check Condition \eqref{item2}. Fix a cycle $I$,  which we enumerate as $l_{1}<l_{2}<\dots< l_{q}$.
  	Assume that we are given $\bar{k}\in\mathcal{X}$ such that $I(\bar{k})=I$ and the equality $D_{u_{i}}(\bar{k})=4\lvert L_{i}\cap I\rvert$ is achieved for every $1\leq i\leq \lvert V\rvert$.
  	
  	This can only take place if $D_{v_{l_{j}}}(\bar{k})=4$ for all $1\leq j\leq q$, which by virtue of Remark \ref{obs_leading} is equivalent to the existence of $w_{j},w'_{j}\in \link_{\overline\Gamma}(u_{l_{j}})$, for all $1\leq j\leq q$, such that
  	\begin{equation}\label{equation_klkl1}
  		(k_{l_{j}},k_{l_{j}+1})=\begin{cases}
  			(r^{1}_{v_{l_{j}},w_{j}},r^{2}_{v_{l_{j}},w'_{j}})&\quad\text{if }v_{l_{j}}\in V_{0};\\
  			(r^{2}_{v_{l_{j}},w_{j}},r^{1}_{v_{l_{j}},w'_{j}})&\quad\text{if }v_{l_{j}}\in V_{1}.
  		\end{cases}
  	\end{equation}
  	Recall that by Lemma \ref{jump set} we have $k_{l}=k_{l+1}$ for any $l\nin I(\bar{k})$. In particular, $k_{l+1}=k_{l'}$ for cyclically consecutive $l\sq l'$ in $I$, and thus Equation \eqref{equation_klkl1} can only hold for $l_{j-1},l_{j},l_{j+1}$ (notice that this triple might contain repetitions) if $w_{j}=v_{l_{j-1}}$ and $w'_{j}=v_{l_{j+1}}$.
  	It follows that the equalities $D_{u_{i}}(\bar{k})=4\lvert L_{i}\cap I(\bar{k})\rvert$ for all $1\leq i\leq m$ can take place only for a unique value $\bar{k}=\bar{k}^{*}$ defined as follows. For any $1\leq j\leq q$ and any $i$ in the (cyclic) interval $[l_{j},l_{j+1})$, where $j+1$ is taken mod $q$, we have:
  	\begin{equation}
  		k^{*}_{i}=\begin{cases}
  			r^{1}_{v_{l_{j}},v_{l_{j+1}}}&\quad\text{if }v_{l_{j}}\in V_{0};\\
  			r^{2}_{v_{l_{j}},v_{l_{j+1}}}&\quad\text{if }v_{l_{j}}\in V_{1}.
  		\end{cases}
  	\end{equation}
Notice that this assignment is consistent since $l_{1},l_{2},\dots, l_{q}$ is a cycle, and therefore, $\overline\Gamma$ being bipartite, must alternate between vertices in $V_0$ and $V_1$.
  	Clearly $\bar{k}^{*}\in\mathcal{X}$, since for any value of $i$ either $k^{*}_{i}=k^{*}_{i+1}$ or both $k_{i}^{*}$ and $k^{*}_{i+1}$ lie in $J_{v}$ for some $v$. The equalities $D_{u_{i}}(\bar{k}^{*})=4\lvert L_{i}\cap I\rvert$ are satisfied for all $1\leq i\leq m$ because the sequence $(v_{l_{j}})_{1\leq j\leq q}$ must alternate between $V_{0}$ and $V_{1}$, as $\bar{\Gamma}$ is bipartite.
  \end{proof}

  \begin{proposition}\label{proposition_uniqueness}
  	Let $\bar{k}\in\mathcal{X}$ with $D(\bar{k})$ $\prec$-maximal. Then $I(\bar{k})=K_{\lvert V\rvert}$ and $D(\bar{k})\neq 0$. In particular, there is a unique $\bar{k}\in\mathcal{X}$ with $D(\bar{k})$ $\prec$-maximal.
  \end{proposition}
  \begin{proof}
  	We prove, by induction on $n$, that $I(\bar{k})\cap(\bigcup_{i=0}^{n}L_{i})=K_{n}$.
  	For $n=0$ the result is trivial, as $I(\bar{k})\cap(\bigcup_{i=1}^{0}L_{i})=\emptyset=K_0$.
  	Suppose the result is satisfied for $n>0$.
  	By Lemma \ref{jump set} we know that $I(\bar{k})$ is a precycle, which together with the induction hypothesis $I(\bar{k})\cap(\bigcup_{i=0}^{n}L_{i})=K_{n}$ and Lemma \ref{optimality} implies that $I(\bar{k})\cap L_{n+1}\subseteq K_{n+1}$, and therefore that
  	\begin{equation}\label{eq_inclusion}
  		I(\bar{k})\cap L_{n+1}\subseteq M_{n+1}=K_{n+1}\cap L_{n+1}.
  	\end{equation}
  	Therefore, by Lemma \ref{upper bound}, for any $1\leq i\leq n+1$ we have that
  	\begin{equation}\label{last_equation}
  		D_{u_{i}}(\bar{k})\leq 4\lvert I(\bar{k})\cap L_{i}\rvert\leq 4\lvert  K_{n+1}\cap L_{i}\rvert\leq D_{u_{i}}(\bar{k}_0),
  	\end{equation}
  	where $\bar{k}_0$ is the unique element in $\mathcal{X}$ associated to the cycle $K_{\lvert V\rvert}$ by Lemma \ref{upper bound}\eqref{item2}.
  	As $D(\bar{k})$ is $\prec$-maximal, the inequalities of Equation \eqref{last_equation} must in fact be equalities, and therefore $\lvert I(\bar k)\cap L_{n+1}\rvert=\lvert  K_{n+1}\cap L_{n+1}\rvert$. As this cardinality is finite, Equation \eqref{eq_inclusion} implies that
  	$I(\bar{k})\cap L_{n+1}= K_{n+1}\cap L_{n+1}$. As $K_{n+1}\cap L_{n+1}=M_{n+1}$ by Lemma \ref{lemma_equality}, we have that $I(\bar{k})\cap L_{n+1}= M_{n+1}$. Therefore
  	\begin{equation*}
  		\begin{split}
  			I(\bar k)\cap\Bigl(\bigcup_{i=0}^{n+1}L_i\Bigr)&=\Bigl(I(\bar k)\cap\bigl(\bigcup_{i=0}^{n}L_i\bigr)\Bigr)\cup \Bigl(I(\bar k)\cap K_{n+1}\Bigl)\\
  			&=K_n\cup M_{n+1}=K_{n+1},
  		\end{split}
  	\end{equation*}
  	and the induction is completed.
  	As $I(\bar{k})\cap \Bigl(\bigcup_{i=0}^{\lvert V\rvert}L_{i}\Bigr)=I(\bar{k})$, applying the just-proved equality to $n=\lvert V\rvert$ we conclude that $I(\bar{k})=K_{\lvert V\rvert}$.

  	\smallskip
  	By Lemma \ref{precycle} we know that $K_{\lvert V\rvert}$ is a cycle.
  	Therefore, by Lemma \ref{upper bound}\eqref{item2} such a $\bar k\in\mathcal{X}$ is unique.
  	Since $D_{u_1}(\bar k)=4\lvert I(\bar k)\cap L_1\rvert\neq 0$,
  	the corresponding monomial is not a constant.
  \end{proof}

  We are now ready to prove Proposition \ref{main lemma}.
  \begin{thmA}
  	\emph{Given a graph $\Gamma$ with $\overline\Gamma$ bipartite and any non-parabolic block $g\in\Gamma\mathcal{G}$, we have $tr\bigl(\lambda(g)\bigr)\notin \R$.}
  \end{thmA}
  \begin{proof}
  	Let $g\in \Gamma\mathcal{G}$ be a non-parabolic, single block element. As per Equation~\eqref{equation_trace},  $tr\bigl(\lambda(g)\bigr)=\sum_{\bar k\in \mathcal{X}}P(\bar k)$. By Proposition \ref{proposition_uniqueness} there exists a unique $\bar{k}^{*}\in\mathcal{X}$ with $D(\bar{k}^{*})$ $\prec$-maximal among the values of $D$ on $\mathcal{X}$.
  	Notice that, if $\mu$ denotes the leading monomial of $P(\bar k)$, then $\mu\notin R$ since $D(\bar{k}^{*})\neq\bar 0$.
  	
  	Uniqueness of $\bar k^*$, together with Lemma \ref{monomial domination}, implies that
  	for any $\bar{k}\in\mathcal{X}\setminus\{\bar{k}^{*}\}$
  	the polynomial $P(\bar{k})$ cannot contain any non-zero scalar multiple of $\mu$.  Therefore, $\mu$ does not cancel in the expression $\sum_{\bar k\in \mathcal{X}}P(\bar k)$ and thus $tr\bigl(\lambda(g)\bigr)\notin \R$.
  \end{proof}
  \begin{remark}
  	One can show that $tr(\lambda(g))\notin R$ for any non-parabolic $g$, even when it decomposes as a product of more than one block.
  \end{remark}

\end{document}